\documentclass[11pt]{article}

\usepackage[latin1]{inputenc}
\usepackage[T1]{fontenc}
\usepackage{amsfonts}
\usepackage{graphicx}
\usepackage{float}
\usepackage[active]{srcltx}
\usepackage{amsmath,amssymb,amsthm}

\newtheorem{fait}{Fact}[section]
\newtheorem{theoreme}[fait]{Theorem}
\newtheorem{lemme}[fait]{Lemma}

\newtheorem{proposition}[fait]{Proposition}

\newcommand{\lima}{\lim_{\alpha\rightarrow+\infty}}
\newcommand{\limab}{\underset{\alpha\rightarrow+\infty}{\lim}}

\newcommand{\ud}{\mathrm{d}}

\newcommand{\Pp}{\mathbb{P}}

\newcommand{\tW}{\widetilde{\mathcal{W}}}

\newcommand{\tB}{\tau_B}
\newcommand{\tM}{\tau_M}
\newcommand{\ta}{\tau_{X_\alpha}}
\newcommand{\tla}{\sigma_{X_\alpha}}
\newcommand{\tlam}{\tla(e^{\alpha h(\alpha)},m)}
\newcommand{\sam}{\sigma_{X_\alpha^m}(e^{\alpha h(\alpha)},m)} 
\newcommand{\sa}{\sigma_{X_\alpha}(e^{\alpha h(\alpha)})}
\newcommand{\sB}{\sigma_B}
\newcommand{\mx}{x_\alpha}
\newcommand{\ax}{\alpha^{-2}x}
\newcommand{\aK}{\alpha^{-2}K}
\newcommand{\R}{\mathbf{R}}
\newcommand{\Z}{\mathbf{Z}}
\newcommand{\N}{\mathbf{N}}
\newcommand{\egloi}{\stackrel{\mathcal{L}}{=}}
\newcommand{\eglw}{\stackrel{\mathcal{L}_W}{=}}

\newcommand{\Wb}{\overline{W}}

\newcommand{\gW}{g(\alpha)}
\newcommand{\cvgloi}{\stackrel{\mathcal{L}}{\longrightarrow}}
\newcommand{\Wu}{\underline{W}}

\newcommand{\un}{1\!\!1}

\title{Limit law of the local time for Brox's diffusion}
 \author{Pierre Andreoletti $^*$, Roland Diel 
\footnote{Laboratoire MAPMO - C.N.R.S. UMR 6628 - F\'ed\'eration Denis-Poisson, Universit\'e d'Orl\'eans, 
(Orl\'eans France). \newline \vspace{0.1cm}  $\quad$  MSC 2000 60J25; 60J55. \newline \vspace{0.5cm} \textit{Key words :  Diffusion process in Brownian potential, Local time} }}

\begin{document}
\maketitle
\begin{abstract}
 We consider Brox's model: a one-dimensional diffusion in a Brownian potential $W$. 
We show that the normalized local time process $(L(t,m_{\log t}+x)/t,\ x \in\R)$, where $m_{\log t}$ is the bottom of the deepest valley reached by the process before time $t$, behaves asymptotically like a process which only depends on $W$. As a consequence, we get the weak convergence of the local time to a functional of two independent three-dimensional Bessel processes and thus the limit law of the supremum of the normalized local time.  These results are discussed and compared to the discrete time and space case which same questions have been solved recently by N. Gantert, Y. Peres and Z. Shi \cite{GanPerShi}.   
\end{abstract}

\section{Introduction}
\subsection{The model}
Let $\left(W(x),\ x\in\R\right)$ be a c\`adl\`ag and locally bounded real valued process such that $W(0)=0$. A \emph{diffusion process in the environment $W$} is a process $\left(X(t),\ t \in \R^+\right)$ which conditional generator given $W$ is
\begin{displaymath}
	\frac{1}{2}e^{W(x)}\frac{d}{dx}\left( e^{-W(x)}\frac{d}{dx} \right).
\end{displaymath}
Notice that for almost surely differentiable $W$, $(X(t),\ t \in \R^+)$ is the solution of the following stochastic differential equation 
\begin{displaymath}
\left\{ \begin{array}{ll}
\ud X(t)=\ud \beta(t)-\frac{1}{2}W'(X(t))\ud t,\\
X(0)=0.
\end{array} \right.
\end{displaymath}
where $\beta$ is a standard one-dimensional Brownian motion independent of $W$. Of course when $W$ is not differentiable, the previous equation has no rigorous sense.
 
This process, introduced by S. Schumacher  \cite{Schumacher} and T. Brox \cite{Brox}, is usually studied with $W$ a L\'evy process. In fact only a few papers deal with the discontinuous case, see for example P. Carmona \cite{Carmona} or A. Singh \cite{Singh}, and most of the results concern continuous $W$, \textit{i.e.} $(W(x):=B_x-\kappa x/2, x \in \R)$, with $\kappa \in \R^+$ and $B$ a two sided Brownian motion independent of $\beta$.
The case $\kappa>0$ is first studied by K. Kawazu and H. Tanaka \cite{KawTan}, then by H. Tanaka \cite{Tanaka} and Y. Hu, Z. Shi and M. Yor \cite{HuShYo} and more recently by M. Taleb \cite{Taleb0} and A. Devulder \cite{Devulder1,Devulder}. The universal characteristic of this case is the transience, however a wide range of limit behaviors appears depending on the value of $\kappa$, see \cite{HuShYo}.

In this paper we choose $\kappa=0$, $X$ is then recurrent and  \cite{Brox} shows that it is sub-diffusive with asymptotic behavior in $(\log t)^2$. Moreover $X$ has the property, for a given instant $t$, to be localized in the neighborhood of a random point $m_{\log t}$ depending only on $t$ and $W$. The limit law of $m_{\log t}/(\log t)^2$ and therefore of $X_t/(\log t)^2$ has been given independently by H. Kesten \cite{Kesten2} and A. O. Golosov \cite{Golosov1}.
In fact, the aim of H. Kesten and A. O. Golosov was to determine the limit law of the discrete time and space analogous of Brox's model introduced by F. Solomon \cite{Solomon} and then studied by Ya. G. Sinai \cite{Sinai}. This random walk in random environment, usually called Sinai's walk, $(S_n,n \in \N)$ has actually the same limit distribution as Brox's one.

Turning back to Brox's diffusion, notice that H. Tanaka \cite{Tanaka3,Tanaka} obtains a deeper localization and later Y. Hu and Z. Shi \cite{HuShi2} get the almost sure rates of convergence. It appears that these rates of convergence are exactly the same as the ones of Sinai's walk. The question of an invariance principle, that could exist between these two processes rises and remains open (see Z. Shi \cite{Shi1} for a  survey).

This work is devoted to the limit distribution of the local time of $X$. Indeed to the diffusion corresponds a local time process $\left(L_X(t,x),\ t\geq0,x\in\R\right)$ defined by the occupation time formula: $L_X$ is the unique $\Pp$-a.s. jointly continuous process such that for any bounded Borel function $f$ and for any $t\geq0$,
\begin{displaymath}
	\int_0^tf(X_s)\ud s=\int_\R f(x)L_X(t,x)\ud x.
\end{displaymath}
The first results on the behavior of $L_X$ can be found in \cite{HuShi1} and \cite{Shi}. In particular, in \cite{HuShi1} they obtain for any $x \in \R$ fixed, 
\begin{align} 
\frac{\log(L_X(t,x))}{\log t} \cvgloi \min(U,\hat{U}), \ t \rightarrow + \infty
\end{align}
where $U$ and $\hat{U}$ are independent variables uniformly distributed in $(0,1)$ and $\cvgloi$ is the convergence in law. Notice that in the same paper Y. Hu and Z. Shi also prove that Sinai's walk has the same behavior: if we denote by  $L_S(n,x):=\sum_{i=1}^n \un_{S_i=x}$ the local time of $S$ in $x\in \Z$ at time $n$ then $$\frac{\log(L_S(n,x))}{\log n} \cvgloi \min(U,\hat{U}),  \ n \rightarrow + \infty.$$  For previous works on the local time for Sinai's diffusion we refer to the book of P. R\'ev\`esz \cite{Revesz}. 

In this article we show that the normalized local time process $(L(t,x+m_{\log t})/t,\ x \in \R)$ behaves asymptotically as a process which only depends on the environment $W$. We also make explicit the limit law of this process when $t$ goes to infinity,  it involves some 3-dimensional Bessel processes.

Define the supremum of the local time process of $X$: 
\begin{displaymath}
\forall t\geq0,\ L_X^*(t):=\sup_{x\in\R}L_X(t,x).  
\end{displaymath}
As a consequence of our results we show that $L_X^*(t)/t$ converges weakly and determine its limit law. We also find interesting to compare the discrete Theorems of \cite{GanPerShi} with ours, pointing out analogies and differences.

\subsection{Preliminary definitions and results}
First let us describe the probability space where $X$ is defined. It is composed of two Wiener spaces, one for the environment and the other one for the diffusion itself. Let $\mathcal{W}$ be the space of continuous functions $W:\R\rightarrow\R$ satisfying $W(0)=0$ and $\mathcal{A}$ the $\sigma$-field generated by the topology of uniform convergence on compact sets on $\mathcal{W}$. We equip $\left(\mathcal{W},\mathcal{A}\right)$ with Wiener measure $\mathcal{P}$ i.e the coordinate process 
is a "two-sided" Brownian motion. We call \emph{environment} an element of $\mathcal{W}$. We then define the set $\Omega:=C([0;+\infty),\R)$, the $\sigma$-field $\mathcal{F}$ on $\Omega$ generated by the topology of uniform convergence on compact sets and the probability measure $P$ such that the coordinate process on $\Omega$ is a standard Brownian motion.

We denote by $\Pp$ the probability product $\mathcal{P}\otimes P$ on $\mathcal{W}\times\Omega$. We indifferently denote by $W$ and call environment an undetermined element of $\mathcal{W}$ and the first coordinate process on $\mathcal{W}\times\Omega$ (i.e a "two-sided" Brownian motion under $\Pp$ or $P$) similarly $B$ is indifferently an element of $\Omega$ and the second coordinate process on $\mathcal{W}\times\Omega$. Finally $\eglw$ means "equality in law" under $P$, that is under a fixed environment $W$, and $\egloi$ (respectively $\cvgloi$) for an equality (resp. a convergence) in law under $\Pp$. We can now state our first result:

\begin{theoreme}\label{thlimsup}$ $ We have
\begin{align*}
	\frac{L^*_X(t)}{t} \cvgloi \frac{1 }{\int_{-\infty}^{\infty}e^{-R(y)}\ud y} 
\end{align*}
where for any $x\in\R$, $R(x):=R_1(x)\textbf{\emph{1}}_{\left\{x\geq0\right\}}+R_2(-x)\textbf{\emph{1}}_{\left\{x<0\right\}}$, $R_1$ and $R_2$ are two independent 3-dimensional Bessel processes starting at $0$. 
\end{theoreme}
Notice that $\Pp$-a.s., $\int_{-\infty}^{\infty}e^{-R(y)}\ud y< + \infty$. Moreover thanks to Le Gall's Ray-Knight theorem (Proposition 1.1 in \cite{LeGall}), the limit law can be expressed in a simpler way:
\begin{align}
\int_{-\infty}^{\infty}e^{-R(y)}\ud y\egloi 4\tau(1)+4\tilde{\tau}(1)\label{Bes2}
\end{align}
where $\tau(1)$ is the hitting time of 1 by a squared Bessel process of dimension 2 starting from 0 and $\tilde{\tau}(1)$ is an i.i.d copy of $\tau(1)$.  
Now we present the equivalent of this theorem for Sinai's walk of \cite{GanPerShi},
\begin{align}
	\frac{L^*_S(n)}{n} \cvgloi \sup_{x \in \Z} \pi(x) \label{GPS}
\end{align}
where 
\begin{align}
	\pi(x)= \frac{\exp(-Z_{x}) +\exp(-Z_{x-1}) }{2\sum_{y \in \Z}\exp(-Z_{y})},\ x \in \Z, \label{invmes}
\end{align} 
and $Z$ is a sum of i.i.d random variables (with mean zero, strictly positive variance and bounded) null at zero and conditioned to stay positive (see Theorem 1.1 and Section 4 of \cite{GanPerShi} for the exact definition).

The analogy between the local time for $X$ and the local time for $S$ takes place in the fact that both $R$ and $S$ can be obtained from classical diffusion conditioned to stay positive: $R_1$ and $R_2$ are Brownian motions conditioned to stay positive (see \cite{Williams}) and $Z$ is a simple symmetric random walk conditioned to stay positive (see \cite{Bertoin} and  \cite{Golosov}). Note that A. O. Golosov also proved in \cite{Golosov} that $\sum_{y \in \Z}\exp(-Z_{y})<+ \infty$. 
However $Z$ and $R$ have not the same nature, one is discrete the other one continuous, notice also that the increments of $Z$ are bounded (see hypothesis 1.2 in \cite{GanPerShi}), and it is not the case for $R$.

Theorem \ref{thlimsup} is an easy consequence of  an interesting intermediate result (Theorem \ref{cvloi} below). Before introducing that result we need some extra definitions on the environment, these basic notions have been introduced by \cite{Brox} (see also \cite{NevPit}). 
Let $h>0$,\label{defvalley} we say that $W\in\mathcal{W}$ admits a $h$-\emph{minimum} at $x_0$ if there exists $\xi$ and $\zeta$ such that $\xi<x_0<\zeta$ and 
\begin{itemize}
\item $W(\xi)\geq W(x_0)+h$,
\item $W(\zeta)\geq W(x_0)+h$.
\item for any $x\in[\xi,\zeta]$, $W(x)\geq W(x_0)$,
\end{itemize}
Similarly we say that $W$ admits a $h$-\emph{maximum} at $x_0$ if $-W$ admits a $h$-minimum at $x_0$. We denote by $M_h(W)$ the set of $h$-\emph{extrema} of $W$. It is easy to establish that $\mathcal{P}$-a.s. $M_h$ has no accumulation point and that the points of $h$-maximum and of $h$-minimum alternate. Hence there exists exactly one triple $\Delta_h=(p_h,m_h,q_h)$ of successive elements in $M_h$ such that 
\begin{itemize}
 \item $m_h$ and $0$ lay in $[p_h,q_h]$,
\item $p_h$ and $q_h$ are $h$-maxima,
\item $m_h$ is a $h$-minimum.
\end{itemize}
We call this triple the standard $h$-valley  \label{stdval} of $W$ (see Figure \ref{fig1}). 
\begin{figure}[ht]
\begin{center}
\input{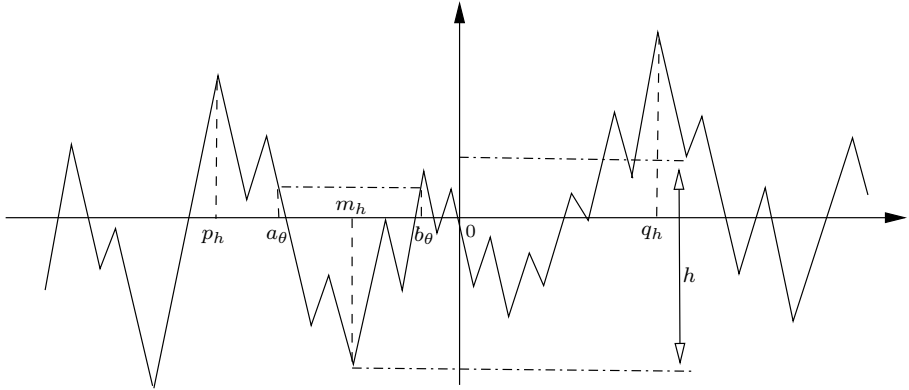} 
\caption{Example of a standard valley} \label{fig1}
\end{center}
\end{figure}

\noindent We can now state our second result:
\begin{theoreme}\label{cvloi}The process $\left(L_X(t,m_{\log t}+x)/t,\ x\in\R\right)$ converges weakly uniformly on compact set to a process $\left(\mathfrak{R}(x),\ x\in\R\right):=e^{-R(x)}/\int_{-\infty}^{\infty}e^{-R}$
where $R$ is the same as in Theorem \ref{thlimsup}.
\end{theoreme}
This result is the analogue of Theorem 1.2 of \cite{GanPerShi}: 
\begin{align}
	\left(\frac{L_S(n,b_n+x)}{n},x \in \Z\right) \cvgloi \left(\pi(x),\ x \in \Z\right) \label{GPS2}
\end{align}
where $\pi(x)$ is given by (\ref{invmes}) and $b_n$ plays the same role for $S$ as $m_{\log t}$ plays for $X$. 

There is a second remark we can make here: H. Tanaka \cite{Tanaka2} shows that the limit  when $t$ goes to infinity of the distribution of the process $(X(t)-m_{\log t})$ converges  to a distribution with a density (with respect to Lebesgue measure) given by $\mathfrak{R}$. There is a natural explanation to understand, at least intuitively, why $\mathfrak{R}$ appears in both limits. For that we need to present another result,  first we introduce $(W_x, x \in \R)$  the \emph{shifted difference of potential},
\begin{align}
	\forall x\in\R,\ W_x(\cdot):=W(x+\cdot)-W(x). \label{shdipo}
\end{align}

\begin{theoreme}\label{limL}Let $K>0$, $r\in(0,1)$. For all $\delta>0$,
\begin{displaymath}
\lima \Pp\left(\sup_{-K\leq x\leq K}\left|\frac{L_X(e^{\alpha},m_\alpha+x)}{e^{\alpha}}\frac{\int_{a_{\alpha r}}^{b_{\alpha r}}e^{-W_{m_\alpha}(y)}\ud y}{e^{-W_{m_\alpha}(x)}}-1\right|\leq \delta \right)=1
\end{displaymath}
where for any $\theta>0$,
\begin{align*}
	& a_{\theta}=a_{\theta}(W_{m_\alpha}):=\sup\left\{x\leq0/W_{m_\alpha}(x)\geq \theta\right\}\textrm{ and} \\
& b_{\theta}=b_{\theta}(W_{m_\alpha}):=\inf\left\{x\geq0/W_{m_\alpha}(x)\geq \theta\right\}, 
\end{align*}
\textrm{see also Figure \ref{fig1}}.	
\end{theoreme}

We can now explain why $\mathfrak{R}$ appears in the paper of H. Tanaka \cite{Tanaka2} and the present one. The important term in the last result 
is the following:
$$
\bar{\mathfrak{R}}(\alpha,x):=\frac{e^{-W_{m_\alpha}(x)}}{\int_{a_{\alpha r}}^{b_{\alpha r}}e^{-W_{m_\alpha}(y)}\ud y},
$$
we show in Section \ref{S3.3} that the process $(\bar{\mathfrak{R}}(\alpha,x),\ x \in \R)$  converges weakly when $\alpha$ goes to infinity to $({\mathfrak{R}}(x),\ x \in \R)$, so Theorem \ref{limL} leads to Theorem \ref{cvloi}. We therefore focus on $\bar{\mathfrak{R}}(\alpha,x)$. Note  that $(\bar{\mathfrak{R}}(\alpha,x),\ x\in (a_{\alpha r},b_{\alpha r}))$ can be seen as a local (in time) invariant probability measure for the process $X$. Indeed until the instant $e^{\alpha}$, $(X(s)\equiv X(s,W),\ s \leq e^{\alpha})$ spends, with a high probability, most of its time between the two points $(a_{\alpha r},b_{\alpha r})$. So $X$ can be approached, in some sense, by a simpler process $(\tilde{X}(s)\equiv \tilde{X}(s,\alpha W),s \leq e^{\alpha})$ with the same generator as $X$ but reflected at fixed barriers, $\tilde{a}$ and $\tilde{b}$ (see \cite{Brox} or \cite{Tanaka2} page 159). This new process obviously possesses an invariant probability measure given by $$\tilde{\mu}_{\alpha}(dx):=\tilde{\mathfrak{R}}(\alpha,x)dx:=\frac{e^{-\alpha W_{m_1}(x)}}{\int_{\tilde{a}}^{\tilde{b}}e^{-\alpha W_{m_1}(y)}\ud y}dx.$$ 
And $\tilde{\mathfrak{R}}(\alpha,x)$ is naturally involved in the limit behavior of the normalized local time $L_{\tilde{X}}(e^{\alpha},m_1+\cdot)/e^{\alpha}$ and also in the limit distribution of $\tilde{X}(e^{\alpha})-m_1$.  To turn back to $\bar{\mathfrak{R}}$ a self similarity argument of the environment can be used.
To finish with this discussion it is interesting to notice that for the discrete time model the result in law plays an important role to get the almost sure asymptotic of the limit sup of $L^*_S$. Indeed (\ref{GPS}) (together with  Lemmata 3.1 and 3.2 in \cite{GanPerShi}) leads to 
$$ \limsup_n \frac{L_S^*(n)}{n}=\textrm{const} \in ]0,+\infty[,\ \Pp\text{-a.s.}$$
and the constant is known explicitly.  
For Brox's model, $L_X^*(t)$ is possibly larger than $t$ and, in fact, if we use a similar argument than \cite{GanPerShi}, then the limit in law (Theorem \ref{cvloi}) only implies $$ \limsup_t \frac{L^*_X(t)}{t}=+\infty,\ \Pp\text{-a.s.}$$ 
which is weaker than the result of Z. Shi \cite{Shi}: $$ \limsup_{t\to\infty} \frac{L^*(t)}{t\log \log \log t} \geq \textrm{const} \in ]0,+\infty] , \ \Pp\text{-a.s.}$$

\subsection{Representation of the diffusion in potential $W$}
In this section we recall basic definitions and tools traditionally used to study diffusion in random environment. Define the stochastic process 
\begin{equation}
\begin{array}{rcl}
X:\mathcal{W}\times\Omega &\longrightarrow& \Omega\\
(W,B)&\longmapsto& S_W^{-1}\circ B\circ T_W^{-1}
\end{array}\label{eqx}
\end{equation}
where
\begin{equation}
 	\forall x \in \R,\ S_W(x):=\int_0^xe^{W(y)}\ud y, \label{eqS}
\end{equation}
and
\begin{equation}
\forall t\geq 0,\ T_W(t):=\int_0^te^{-2W(S_W^{-1}(B(s)))}\ud s.\label{eqT}
\end{equation}
As Brox points out in \cite{Brox}, the standard diffusion theory implies that this process is under $\Pp$ a diffusion in Brownian environment. To simplify notations, we write when there is no possible mistake $S$ and $T$ for respectively $S_W$ and $T_W$. Using Formula $(\ref{eqx})$, we easily obtain that for any $x\in\R$ and $t\geq0$,
\begin{align}
	L_X(t,x)=e^{-W(x)}L_B(T^{-1}(t),S(x)) \label{tsc}
\end{align}
where $L_B$ is the local time process of the Brownian motion $B$. 

Brox \cite{Brox} noticed also that it is more convenient to study the asymptotic behaviors of the process $X_\alpha(W,\cdot):=X(\alpha W,\cdot)$ and of its local time process $L_{X_\alpha}$ instead of $X$ and $L_X$. For all $x\in\R$, we denote 
$$W^\alpha(x):=\frac{1}{\alpha}W(\alpha^2x).$$ 
For $\alpha>0$ fixed, there is a link between  $X_{\alpha}$ and $X$ given by:
\begin{lemme}\label{egx}
For each $W\in\mathcal{W}$ and $\alpha>0$,
\begin{align*}
& \left(X_{\alpha}(W^\alpha,t)\right)_{t\geq0}:= \left(X(\alpha W^\alpha,t)\right)_{t\geq0} \eglw \left(\frac{1}{\alpha^2}X(W,\alpha^4t)\right)_{t\geq0},\\
& \left(L_{X_\alpha(W^\alpha,\cdot)}(t,x)\right)_{t\geq0,x\in\R} \eglw \left(\frac{1}{\alpha^2}L_{X(W,\cdot)}(\alpha^4t,\alpha^2x)\right)_{t\geq0,x\in\R}.
\end{align*}
\end{lemme}

We do not give any detail on the proof of this lemma, the first relation can be found in Brox (see \cite{Brox}, Lemma 1.3) and the second is a straightforward consequence of the first one. Formulas $(\ref{eqS})$, $(\ref{eqT})$ and $(\ref{eqx})$ have their equivalent for $X_\alpha$:
\begin{equation}
\forall t \geq 0,\ X_\alpha(t)=S_\alpha^{-1}(B(T_\alpha^{-1}(t)))\label{eqxa}
\end{equation}
where
\begin{equation}
	\forall x \in \R,\ S_\alpha(x):=S_{\alpha W}(x)=\int_0^xe^{\alpha W(y)}\ud y
\end{equation}
and
\begin{equation}
\forall t\geq 0,\ T_\alpha(t):=T_{\alpha W}(t)=\int_0^te^{-2\alpha W(S_\alpha^{-1}(B(s)))}\ud s.
\end{equation}
And for the local time we have,
\begin{equation}
	\forall t\geq 0,\ \forall x \in \R,\ L_{X_\alpha}(t,x)=e^{-\alpha W(x)}L_B(T_\alpha^{-1}(t),S_\alpha(x)).\label{eqLxa}
\end{equation}

The rest of the paper is organized  as follows: in the first part of Section $\ref{S2}$ we get the asymptotic of the local time at specific random time, which is the inverse of the local time at $m_{\log t}$, in Section $\ref{2.2}$ the asymptotic of the inverse of the local time itself is studied. Note that Sections \ref{2.1} and \ref{2.2} can be read independently. Propositions \ref{limLT} and \ref{limT} of Sections \ref{S2} are the key results to get  Theorem \ref{limL} proved at the beginning  of Section \ref{S3}.  In the second and third subsection of Section \ref{S3},  Theorems \ref{thlimsup} and \ref{cvloi} are proved.

\section{Asymptotics for the local time and its inverse \label{S2}}

We begin with some definitions that are used all along the paper.
For any process $M$, we define the following stopping times with the usual convention $\inf\emptyset=+\infty$, 
\begin{align}
& \forall x\in\R,\ \tM(x) :=\inf\{t\geq0/M(t)=x\}, \\
& \forall x\in\R,\ \forall r\geq0,\ \sigma_M(r,x):=\inf\{t\geq0/L_M(t,x)\geq r\}. \label{invloctim} 
\end{align}
We define for any $W\in\mathcal{W}$ and for all $x,y\in\R$,
$$\Wb(x,y):=\left\{ \begin{array}{ll}
\sup_{[x,y]}W & \textrm{if $y\geq x$},\\
\sup_{[y,x]}W & \textrm{if $y<x$},
\end{array}
\right.
$$
and
$$\Wu(x,y):=\left\{ \begin{array}{ll}
\inf_{[x,y]}W & \textrm{if $y\geq x$},\\
\inf_{[y,x]}W & \textrm{if $y<x$},
\end{array}
\right.
$$
they represent respectively the maximum and the minimum of $W$ between $x$ and $y$.
 Finally, we introduce the process starting in $x\in\R$,
\begin{displaymath}
\left(X_\alpha^x(W,t),\ t\geq0\right):=\left(x+X_\alpha(W_x,t),\ t\geq0\right)=\left(x+X(\alpha W_x,t),\ t\geq0\right)
\end{displaymath}
where $W_x$ is the \emph{shifted difference of potential} (see (\ref{shdipo})). Notice that we have the equivalent of (\ref{eqxa}): 
\begin{equation}
\forall t \geq 0, X_\alpha^x(t)=x+(S^x_\alpha)^{-1}(B((T^x_\alpha)^{-1}(t)))\label{eqxax}
\end{equation}
where $S^x_\alpha:=S_{\alpha W_x}$, $T^x_\alpha:=T_{\alpha W_x}$ and it is easy to establish that for a fixed $W\in\mathcal{W}$, $X_\alpha^x$ is a strong Markov process.

\subsection{Asymptotic behavior of $L_{X_\alpha}$ at time $\sigma_{X_\alpha}(e^{\alpha h(\alpha)},m)$ \label{2.1}}
In this first sub-section we study the asymptotic behavior of the local time at the inverse of the local time  in $m:=m_1$, recall that $m_1$ is the coordinate of the bottom of the standard valley defined page \pageref{stdval}.
\begin{proposition}\label{limLT} Let $K>0$, $W\in\mathcal{W}$ and let $h$ be a function such that $\limab h(\alpha)=1$. Then, for all $\delta>0$,
\begin{displaymath}
\lima P\left(\sup_{-K\leq x\leq K}
\left|\frac{L_{X_\alpha}(\tlam,m+\alpha^{-2}x)}{e^{\alpha h(\alpha)-\alpha W_m(\alpha^{-2}x)
}}-1\right|\leq \delta \right)=1.
\end{displaymath}
\end{proposition}
\begin{proof}
For simplicity, we assume 
that $m=m_1(W)\geq0$
and to lighten notations, we denote for all $x\in [-K,K],\ \mx:=m+\alpha^{-2}x$. 

 The proof is based on the decomposition of the local time into two terms, the first one is the contribution of the local time in $\mx$ before $\ta(m)$ (the first time $X_\alpha$ hits $m$) and the second one is the contribution of the local time between $\ta(m)$ and $\tlam$ (the inverse of the local time in $m$):
\begin{align*}
L_{X_\alpha}(\tlam,\mx)&=L_{X_\alpha}(\ta(m),\mx)\\
&+\left(L_{X_\alpha}(\tlam,\mx)-L_{X_\alpha}(\ta(m),\mx)\right).
\end{align*}
We treat these two terms in Lemmata \ref{eqLtm1} and \ref{eqLtl} below.  Lemma $\ref{eqLtm1}$ states that, asymptotically, the local time in a point $\mx$ until the process reaches $m$ is negligible compared to $e^{\alpha h(\alpha)-\alpha W_m(\alpha^{-2}x)}$. 
Thanks to the strong Markov property for $X_\alpha$, it remains to study the asymptotic behavior of
$$\left(L_{X_\alpha^m}(\sam,m+\ax)\right)_{-K\leq x\leq K}$$ when $\alpha$ goes to infinity, this is what is done in  Lemma \ref{eqLtl}: it says that the local time of $X_\alpha^m$ at $\mx$ within the interval of time $[0,\tlam]$ is of the order of $e^{\alpha h(\alpha)-\alpha W_m(\alpha^{-2}x)}$. 
\end{proof}
We now state and prove
\begin{lemme}\label{eqLtm1}
Let $W\in\mathcal{W}$, for any $\delta>0$,
\begin{displaymath}
\lima P\left(\sup_{-K\leq x\leq K}\frac{L_{X_\alpha}(\ta(m),\mx)}{e^{\alpha h(\alpha)-\alpha W_m(\alpha^{-2}x)}}\leq \delta \right)=1.
\end{displaymath}
\end{lemme}
\begin{proof}
First, as we have assumed that $m\geq 0$, for all $x>0$ $L_{X_\alpha}(\ta(m),\mx)=0$ so we only have to consider  non positive $x$. Notice also that for all $x\in[-K,0]$, $L_{X_\alpha}(\ta(m-\alpha^{-2}K),\mx)=0,$
therefore
\begin{align}
L_{X_\alpha}(\ta(m),\mx)=L_{X_\alpha}(\ta(m),\mx)-L_{X_\alpha}(\ta(m-\alpha^{-2}K),\mx). \label{eq1}
\end{align}
Let $\kappa_\alpha=m-\alpha^{-2}K$. Thanks to (\ref{eq1}) and the strong Markov property for $X_\alpha$, we only need to prove that
\begin{equation}
\lima P\left(\sup_{-K\leq x\leq 0}\frac{L_{X_\alpha^{\kappa_\alpha}}(\tau_{X_\alpha^{\kappa_\alpha}}(m),\mx)}{e^{\alpha h(\alpha)-\alpha W_m(\alpha^{-2}x)}}\leq \delta \right)=1.\label{eqLtm}
\end{equation}
It follows from $(\ref{eqxax})$ with $x=\kappa_\alpha$ that
\begin{align*}
 \tau_{X_\alpha^{\kappa_\alpha}}(m)=\tau_{X_\alpha(W_{\kappa_\alpha},\cdot)}(\alpha^{-2}K)=T_\alpha^{\kappa_\alpha}(\tB(S^{\kappa_\alpha}_\alpha(\alpha^{-2}K))),\label{eqtam}
\end{align*}
so according to $(\ref{eqLxa})$, we have for all $x\in\R$,
\begin{align*}
	&L_{X_\alpha^{\kappa_\alpha}}(\tau_{X_\alpha^{\kappa_\alpha}}(m),\mx)=\\
	&e^{-\alpha W_{\kappa_\alpha}(\alpha^{-2}(x+K))}L_B(\tB(S_\alpha^{\kappa_\alpha}(\alpha^{-2}K)),S_\alpha^{\kappa_\alpha}(\alpha^{-2}(x+K))).
\end{align*}
The classic scaling property of the local time of Brownian motion given by
\begin{equation}
\forall \lambda>0,\ \forall y_1>0,\ \left(\lambda L_B(\tB(y_1),y)\right)_{y\in\R}\eglw\left(L_B(\tB(\lambda y_1),\lambda y)\right)_{y\in\R},\label{eqsca1}
\end{equation}
yields that the processes $\left(L_{X_\alpha^{\kappa_\alpha}}\Big(\tau_{X_\alpha^{\kappa_\alpha}}(m),\mx\Big)\right)_{x\in\R}$
and
\begin{displaymath}
\left(S_\alpha^{\kappa_\alpha}(\alpha^{-2}K)e^{-\alpha W_{\kappa_\alpha}(\alpha^{-2}(x+K))}L_B\Big(\tB(1),s_\alpha( \alpha^{-2} (x+K) )\Big)\right)_{x\in\R}
\end{displaymath}
are equal in law, where $s_\alpha(z):=S_\alpha^{\kappa_\alpha}(z)/S_\alpha^{\kappa_\alpha}(\alpha^{-2}K)$.

We claim that for all $x\in[-K,0]$
\begin{align*}
	&S_\alpha^{\kappa_\alpha}(\alpha^{-2}K)e^{-\alpha W_{\kappa_\alpha}(\alpha^{-2}(x+K))}L_B(\tB(1),s_\alpha(\alpha^{-2}(x+K)))\\
	&\leq 
	\alpha^{-2}Ke^{\alpha(\Wb_m(-\alpha^{-2}K,0)-W_m(\alpha^{-2}x))}\sup_{y\leq1}L_B(\tB(1),y).
\end{align*}
Indeed
\begin{displaymath}
 S_\alpha^{\kappa_\alpha}(\alpha^{-2}K)=\int_0^{\frac{K}{\alpha^2}}e^{\alpha W_{\kappa_\alpha}(y)}\ud y\leq\frac{K}{\alpha^2}e^{\alpha\Wb_{\kappa_\alpha}(0,\alpha^{-2}K)}
\end{displaymath}
and for all $x\in[-K,0]$,
\begin{displaymath}
 \Wb_{\kappa_\alpha}(0,\alpha^{-2}K)-W_{\kappa_\alpha}(\alpha^{-2}(x+K))=\Wb_m(-\alpha^{-2}K,0)-W_m(\alpha^{-2}x)
\end{displaymath}
and $s_\alpha(\alpha^{-2}(x+K))\leq 1$.

Assembling the above estimates, we get for any $\delta>0$,
\begin{align*}
&P\left(\sup_{-K\leq x\leq 0}\frac{L_{X_\alpha^{\kappa_\alpha}}(\tau_{X_\alpha^{\kappa_\alpha}}(m),\mx)}{e^{\alpha h(\alpha)-\alpha W_m(\alpha^{-2}x)}}\leq \delta \right)\\
&\geq P\left(\alpha^{-2}Ke^{\alpha \Wb_m(-\alpha^{-2}K,0)
-\alpha h(\alpha)}\sup_{y\leq1}L_B(\tB(1),y)\leq \delta\right).
\end{align*}
As
$\limab\Wb_m(-\alpha^{-2}K,0)=0$, $\limab h(\alpha)=1$ and $\underset{y\leq1}{\sup}L_B(\tB(1),y)<\infty$ $P$-a.s., the right hand side of the last inequality tends to $1$ as $\alpha$ goes to infinity and then $(\ref{eqLtm})$ is proved together with the lemma.
\end{proof}
We move to the proof of the second lemma
\begin{lemme}\label{eqLtl}
Let $W\in\mathcal{W}$. For any $\delta>0$,
\begin{displaymath}
\lima P\left(\sup_{-K\leq x\leq K}\left|\frac{L_{X_\alpha^m}(\sam,m+\ax)}{e^{\alpha h(\alpha)-\alpha W_m(\ax)}}-1\right|\leq \delta \right)=1.
\end{displaymath}
\end{lemme}
\begin{proof}
For simplicity we denote for any process $M$, $\sigma_M(r):=\sigma_M(r,0)$, we also  assume without loss of generality that $m=0$, Lemma \ref{eqLtl} can therefore be rewritten in the following way:
\begin{equation}
\lima P\left(\sup_{-K\leq x\leq K}\left|\frac{L_{X_\alpha}(\sa,\ax)}{e^{\alpha h(\alpha)-\alpha W(\ax)}}-1\right|\leq \delta \right)=1.\label{eqLtlap}
\end{equation}
Like for $\tau$ in Lemma \ref {eqLtm1}, we easily get that $\sigma_{X_\alpha}(t)=T_\alpha(\sigma_B(t))$. Thus formula $(\ref{eqLxa})$ together with the scale invariance for the local time of Brownian motion
\begin{equation}
\forall r>0, \forall\lambda>0, \left(\lambda L_B(\sigma_B(r),y)\right)_{y\in\R}\eglw\left(L_B(\sigma_B(\lambda r),\lambda y)\right)_{y\in\R}\label{eqsca2}
\end{equation}
yields
\begin{displaymath}
\left(L_{X_\alpha}(\sa,\ax)\right)_{x\in\R}\eglw\left(e^{\alpha h(\alpha)-\alpha W(\ax)}L_B(\sigma_B(1),\widetilde{s}_\alpha(\ax))\right)_{x\in\R},
\end{displaymath}
where $\widetilde{s}_\alpha(\ax):=S_\alpha(\ax)e^{-\alpha h(\alpha)}$.
Let $K_\alpha:=\aK e^{\alpha \overline{W}(-\aK,\aK)-\alpha h(\alpha)}$, we have for all $x\in[-K,K]$, $-K_\alpha\leq \widetilde{s}_\alpha(\ax)\leq K_\alpha$. By collecting what we did above we get 
for any $\delta>0$,
\begin{align*}
&P\left(\sup_{-K\leq x\leq K}\left|\frac{L_{X_\alpha}(\sa,\alpha^{-2}x)}{e^{\alpha h(\alpha)-\alpha W(\ax)}}-1\right|\leq \delta \right)\\
&=P\left(\sup_{-K\leq x\leq K}\left|L_B(\sigma_B(1),\widetilde{s}_\alpha(\ax))-1\right|\leq \delta \right)\\
&\geq P\left(\sup_{-K_\alpha\leq y\leq K_\alpha} |L_B(\sigma_B(1),y)-1|\leq\delta\right).
\end{align*}
Moreover $\limab K_\alpha=0$ and $y\rightarrow L_B(\sigma_B(1),y)$ is continuous at $0$, so (\ref{eqLtlap}) and the lemma are proved.
\end{proof}

\subsection{Asymptotic behavior of $\sigma_{X_\alpha}(e^{\alpha h(\alpha)},m)$ \label{2.2}}
Before stating the main proposition of this section, we need a preliminary result on the random environment which gives precisions on the standard $h$-valley $\Delta_h(W)=(p_h,m_h,q_h)$ defined Section \ref{stdval}. We denote
\begin{displaymath}
 W^\#(x,y):=\max_{[x\wedge y,x\vee y]}\left(W(z)-\Wu(x,z)\right)
\end{displaymath}
where $x\vee y= \max(x,y)$ and $x\wedge y=\min(x,y)$. Notice that $W^\#$ represents the largest barrier of potential the process has to cross to go from $x$ to $y$.
We call \emph{depth of the valley} $\Delta_h(W)$ the quantity
\begin{displaymath}
	D(\Delta_h(W)):=(W(p_h)-W(m_h))\wedge(W(q_h)-W(m_h))
\end{displaymath}
and \emph{inner directed ascent} the quantity
\begin{displaymath}
	A(\Delta_h(W)):=W^\#(p_h,m_h)\vee W^\#(q_h,m_h).
\end{displaymath}
Note that the above notions have already been introduced by Sinai \cite{Sinai}, Brox \cite{Brox}, and Tanaka \cite{Tanaka}. According to Brox, we have the following
\begin{lemme}
There exists a subset $\tW$ of $\mathcal{W}$ of $\mathcal{P}$-measure $1$ such that for any $W\in\tW$, the standard $1$-valley $\Delta_1(W):=(p_1,m_1,q_1)$ satisfies $A(\Delta_1(W))<1<D(\Delta_1(W))$.\end{lemme}
Throughout this section we write $p,m,q,D$ and $A$ for respectively $m_1(W)$, $p_1(W)$, $q_1(W)$, $D(\Delta_1(W))$ and $A(\Delta_1(W))$.
  \begin{figure}[ht]
\begin{center}
\input{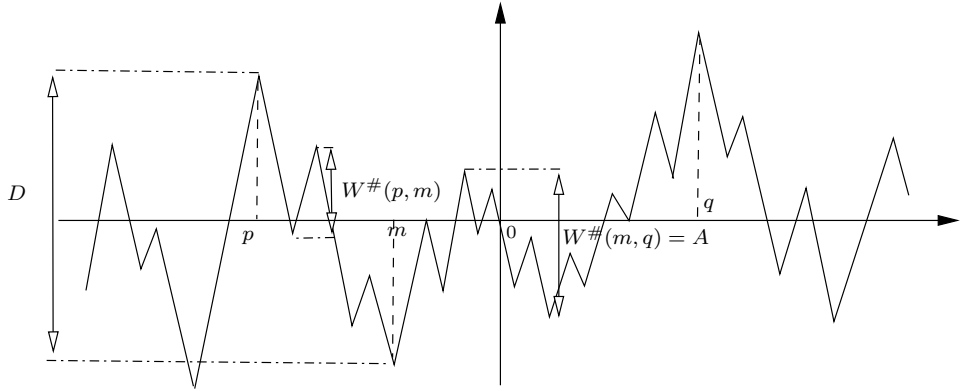} 
\caption{Example of 1-standard valley with its depth and its inner directed ascent.} \label{fig2}
\end{center}
\end{figure}

We can now state the main result of this section:
\begin{proposition}\label{limT}Let $W\in\tW$, $r\in(0,1)$ and $h$ be a function such that $\limab h(\alpha)=1$, then for all $\delta>0$,
\begin{displaymath}
\lima P\left(\left|\frac{\tlam}{e^{\alpha h(\alpha)}\int_{a_{r}}^{b_{r}}e^{-\alpha W_m(x)}\ud x}-1\right|\leq\delta \right)=1
\end{displaymath}
where $a_{r}$ and $b_{r}$ are defined in Theorem \ref{limL}.
\end{proposition}
To lighten notations, in the rest of the paper we use the notation	
\begin{displaymath}
\gW:=\int_{a_{r}}^{b_{r}}e^{-\alpha W_m(y)}\ud y,\ \alpha>0.
\end{displaymath}
\begin{proof}
Assume that $m>0$, we get the other case by reflection, note that we work with a fixed $W$ which belongs to $\tW$. We follow the same steps of the proof of Proposition \ref{limLT}: $\tlam$ is decomposed into two terms,
\begin{displaymath}
\tlam=\ta(m)+\left(\tlam-\ta(m)\right).
\end{displaymath}
The first one $\ta(m)$ is treated in Lemma \ref{eqtm}, we show that its contribution is negligible comparing to $g(\alpha)e^{\alpha(\alpha)}$. Then thanks to the strong Markov property, it is enough to prove that $\tlam/g(\alpha)e^{\alpha h(\alpha)}$ converges to $1$ in $P$ probability, this is what Lemma \ref{eqtl} says.
\end{proof}

Let us state and prove the first lemma
\begin{lemme}\label{eqtm}
Let $W\in\tW$. For any $\delta>0$,
\begin{displaymath}
\lima P\left(\frac{\ta(m)}{\gW e^{\alpha h(\alpha)}}\leq \delta \right)=1.
\end{displaymath}
\end{lemme}
\begin{proof}
This proof has the same outline as the proof of the point $(i)$ of Lemma $3.1$ in \cite{Brox}, however because of some slight differences and for completeness we give some details. 

By definition of the local time together with (\ref{eqLxa}) and (\ref{eqsca1}), the hitting time of $m$ can be written
\begin{align}
	\ta(m)=&\int_{-\infty}^{m}L_{X_\alpha}(\ta(m),z)\ud z, \nonumber\\
=&\int_{-\infty}^{m}e^{-\alpha W(z)}L_B(\tB(S_\alpha(m)),S_\alpha(z))\ud z, \nonumber \\
\eglw& S_\alpha(m)\int_{-\infty}^{m}e^{-\alpha W(z)}L_B(\tB(1),\hat{s}_\alpha(z))\ud z \label{andromaque}
\end{align}
where $\hat{s}_\alpha(z):=S_\alpha(z)/S_\alpha(m)$. Let $n:=\underset{[0,m]}{\textrm{argmax }}W$, we denote 
\begin{align*}
 I_{\alpha,1}&:=S_\alpha(m)\int_{-\infty}^{n}e^{-\alpha W(z)}L_B(\tB(1),\hat{s}_\alpha(z))\ud z, \\
 I_{\alpha,2}&:=S_\alpha(m)\int_{n}^{m}e^{-\alpha W(z)}L_B(\tB(1),\hat{s}_\alpha(z))\ud z,
\end{align*}
and Formula $(\ref{andromaque})$ can be rewritten
\begin{equation}
\ta(m)\eglw I_{\alpha,1}+I_{\alpha,2}.\label{I1I2}
\end{equation}
The rest of the proof consists essentially in finding an upper bound for $I_{\alpha,1}$ and $I_{\alpha,2}$.

\emph{We begin with $I_{\alpha,1}$}: first we prove that, with a probability which tends to $1$ when $\alpha$ goes to infinity, the process $X_\alpha$ does not visit coordinates smaller than $p$, where $p$ is the left vertex of the standard $1$-valley defined page \pageref{defvalley}. Thanks to this, the lower bound in the integral of $I_{\alpha,1}$ will be $p$ and not $-\infty$, \label{defl} the upper bound for $I_{\alpha,1}$ follows immediately. 

 Let us define 
\begin{align}
	l:=\inf\left\{x\leq 0/L_B(\tB(1),x)>0\right\},
\end{align}
we claim that $P\text{-a.s.}$, for $\alpha$ large enough, 
\begin{equation}
\hat{s}_\alpha^{-1}(l)>p.\label{eqsp}
\end{equation}
Indeed 
\begin{displaymath}
	\hat{s}_\alpha^{-1}(l)\geq p\Longleftrightarrow l\geq \hat{s}_\alpha(p)=-\frac{\int_p^0e^{\alpha W(x)}\ud x}{\int_0^me^{\alpha W(x)}\ud x},
\end{displaymath}
moreover Laplace's method gives
\begin{align*}
	\lima\frac{1}{\alpha}\log\int_p^0e^{\alpha W(x)}\ud x&=\Wb(p,0)=W(p)\textrm{ and}\\
	\lima\frac{1}{\alpha}\log\int_0^me^{\alpha W(x)}\ud x&=\Wb(0,m)=W(n)
\end{align*}
so
\begin{displaymath}
	\hat{s}_\alpha(p)=-\exp{(\alpha(W(p)-W(n))+o(\alpha))}.
\end{displaymath}
 Finally, according to the definition of the standard valley, $W(p)>W(n)$, therefore
\begin{displaymath}
	\lima \hat{s}_\alpha(p)=-\infty
\end{displaymath}
and $(\ref{eqsp})$ is true. 

On the event $\left\{\hat{s}_\alpha^{-1}(l)>p\right\}$, we have
\begin{align*}
I_{\alpha,1}&=S_\alpha(m) \int_{p}^{n}e^{-\alpha W(z)}L_B(\tB(1),\hat{s}_\alpha(z))\ud z\\
&\leq S_\alpha(m) (n-p)e^{-\alpha\Wu(p,n)}\max_{x\leq1}L_B(\tB(1),x).
\end{align*}
Moreover,
\begin{align*}
 S_\alpha(m)\leq me^{\alpha \Wb(0,m)}\leq (q-p)e^{\alpha W(n)}
\end{align*}
thus we get the upper bound
\begin{align}
	I_{\alpha,1}\leq(q-p)^2e^{\alpha A}\max_{x\leq1}L_B(\tB(1),x)\label{majI1}
\end{align}
where $A$ is the inner direct ascent of the valley defined at the beginning of this section.

\emph{We continue with $I_{\alpha,2}$}, the main ingredient to get an upper bound in this case is to use the first Ray-Knight theorem which reduces the study of an integral involving a two-dimensional Bessel process: first we rewrite $I_{\alpha,2}$ in the following way
\begin{align}
I_{\alpha,2}=S_\alpha(m)\int_n^me^{-\alpha W(z)}L(\tB(1),1-\bar{s}_\alpha(z))\ud z \label{eq26}
\end{align}
where
$$\bar{s}_\alpha(z):=1-\hat{s}_\alpha(z)=\frac{1}{S_\alpha(m)}\int_z^me^{\alpha W(x)}\ud x.$$
Let $R$ be a two-dimensional Bessel squared process starting at the origin, according to the first Ray-Knight theorem 
$$\left(L_B(\tB(1),1-\bar{s}_\alpha(z))\right)_{z\in[0,m]}\eglw\left(R(\bar{s}_\alpha(z))\right)_{z\in[0,m]} \label{bobby},$$	
 together with the scale invariance $(t^2R(1/t),\ t\in\R_+)\eglw(R(t),\ t\in\R_+)$ we get
\begin{align}
\int_n^me^{-\alpha W(z)}R(\bar{s}_\alpha(z))\ud z\eglw\int_n^m\left\{e^{-\alpha W(z)}\bar{s}_\alpha(z)\right\}\bar{s}_\alpha(z)R(\frac{1}{\bar{s}_\alpha(z)})\ud z. \label{eq27}
\end{align}
We are now able to get a preliminary upper bound for $I_{\alpha,2}$:
\begin{align}
&S_\alpha(m)\int_n^m\left\{e^{-\alpha W(z)}\bar{s}_\alpha(z)\right\}\bar{s}_\alpha(z)R(\frac{1}{\bar{s}_\alpha(z)})\ud z,\nonumber\\
&\leq\max_{n\leq z\leq m}\left[e^{-\alpha W(z)}\int_z^me^{\alpha W(x)}\ud x\right]\int_n^m\bar{s}_\alpha(z)R(\frac{1}{\bar{s}_\alpha(z)})\ud z,\nonumber\\
&\leq(q-p)\exp\left\{\alpha\max_{n\leq z\leq m}\left[-W(z)+\Wb(z,m)\right]\right\}\int_n^m\bar{s}_\alpha(z)R(\frac{1}{\bar{s}_\alpha(z)})\ud z,\nonumber\\
&\leq(q-p)^2e^{\alpha A}\underbrace{\frac{1}{m-n}\int_n^m\bar{s}_\alpha(z)R(\frac{1}{\bar{s}_\alpha(z)})\ud z}_{J_\alpha}.\label{majI2b}
\end{align}
The last inequality comes from the relation $\max_{n\leq z\leq m}\left[-W(z)+\Wb(z,m)\right]=W^\#(n,m)\leq A$. 

According to Jensen's inequality and Fubini's theorem the expectation of $(J_\alpha)^2$ satisfies
\begin{align}
E\left[(J_\alpha)^2\right]&\leq\frac{1}{m-n}\int_n^mE[\bar{s}_\alpha^2(z)(R)^2(\frac{1}{\bar{s}_\alpha(z)})]\ud z,\nonumber\\
&=\frac{1}{m-n}\int_n^m\int_0^{+\infty}\frac{\bar{s}_\alpha(z)^3y^2}{2}e^{-\frac{\bar{s}_\alpha(z)y}{2}}\ud y\ud z=8.\label{eja}
\end{align}
\emph{End of the proof of the lemma}: using $\eqref{I1I2}$, we obtain for all $\alpha>0$
\begin{align*}
	&P\left(\frac{\ta(m)}{\gW e^{\alpha h(\alpha)}}\leq \delta \right)= P\left(\frac{I_{\alpha,1}+I_{\alpha,2}}{\gW e^{\alpha h(\alpha)}}\leq \delta \right)\\
	&\geq P\left(\frac{I_{\alpha,1}}{\gW e^{\alpha h(\alpha)}}\leq \frac{\delta}{2}\textrm{ ; } \hat{s}_\alpha^{-1}(l)\geq p\right)+P\left(\frac{I_{\alpha,2}}{\gW e^{\alpha h(\alpha)}}\leq \frac{\delta}{2}\right)-1.
\end{align*}
For the first term in the above expression, $(\ref{majI1})$ yields
\begin{align}
&P\left(\frac{I_{\alpha,1}}{\gW e^{\alpha h(\alpha)}}\leq \frac{\delta}{2}\textrm{ ; }\hat{s}_\alpha^{-1}(l)\geq p\right)\geq\nonumber\\
&P\left(\max_{x\leq1}L_B(\tB(1),x)\leq G(\alpha)\textrm{ ; }\hat{s}_\alpha^{-1}(l)\geq p\right)\label{pasdidee}
\end{align}
where
\begin{displaymath}
 G(\alpha):=\frac{\delta g(\alpha)}{2(q-p)^2}e^{\alpha (h(\alpha)-A)}.
\end{displaymath}
By Laplace's method we know that $\limab\frac{\log g(\alpha)}{\alpha}=0$, so $g(\alpha)=e^{\circ(\alpha)}$, therefore as $\limab h(\alpha)=1>A$, $G(\alpha)$ tends to infinity when $\alpha$ does. Using that $y\rightarrow L_B(\tB(1),y)$ is $P$-a.s. finite and (\ref{eqsp}) we get that $(\ref{pasdidee})$ tends to $1$ when $\alpha$ goes to infinity.

For the second term, we collect (\ref{eq26}), (\ref{eq27}) and (\ref{majI2b}), we get
\begin{displaymath}
P\left(\frac{I_{\alpha,2}}{\gW e^{\alpha h(\alpha)}}>\frac{\delta}{2}\right)\leq P\left(J_\alpha> G(\alpha)\right),
\end{displaymath}
then by Tchebytchev's inequality and $(\ref{eja})$
\begin{align*}
P\left(J_\alpha> G(\alpha)\right)&\leq\frac{E\left[(J_\alpha)^2\right]}{\left(G(\alpha)\right)^2} 
 \leq \frac{8}{\left(G(\alpha)\right)^2}.
\end{align*}
Using once again that $G(\alpha)$ tends to infinity we get the lemma.
\end{proof}
 Next step is to prove 
\begin{lemme}\label{eqtl}
Let $W\in\tW$. For any $\delta>0$,
\begin{displaymath}
\lima P\left(\left|\frac{\sam}{\gW e^{\alpha h(\alpha)}}-1\right|\leq \delta \right)=1.
\end{displaymath}
\end{lemme}
\begin{proof}
Just like for the proof of Lemma \ref{eqLtl} we assume without loss of generality that $m=0$, as a consequence $\sam=\sa$ and we simply have to establish that:
\begin{align}
\lima P\left(\left|\frac{\sa}{\gW e^{\alpha h(\alpha)}}-1\right|\leq \delta \right)=1.
\label{eqtlm}
\end{align}
In the same way we get $(\ref{andromaque})$, one can prove that
\begin{equation}
 \sa\eglw e^{\alpha h(\alpha)}\int_{-\infty}^{+\infty}e^{-\alpha W(x)}L_B(\sB(1),\widetilde{s}_\alpha(x))\ud x, \label{eglsa}
\end{equation}
recall that $\widetilde{s}_\alpha(y)=S_\alpha(y)e^{-\alpha h(\alpha)}$. The rest of the proof is devoted to estimate the integral and the main difficulty is to get the upper bound.

\emph{We begin with the lower bound}, we easily get that 
\begin{align*}
	\int_{-\infty}^{+\infty}e^{-\alpha W(x)}L_B(\sB(1),\widetilde{s}_\alpha(x))\ud x
	&\geq\int_{a_r}^{b_r}e^{-\alpha W(x)}L_B(\sB(1),\widetilde{s}_\alpha(x))\ud x,\\
	&\geq\inf_{y\in[\widetilde{s}_\alpha(a_{r}),\widetilde{s}_\alpha(b_{r})]}L_B(\sB(1),y)g(\alpha),
\end{align*}
where $a_r$ and $b_r$ are defined at the end of Theorem \ref{limL}, therefore
\begin{align}
&P\left(\frac{\sa}{\gW e^{\alpha h(\alpha)}}\geq1-\delta\right)\geq\nonumber\\ 
&P\left(\inf_{y\in[\widetilde{s}_\alpha(a_{r}),\widetilde{s}_\alpha(b_{r})]}L_B(\sB(1),y)\geq1-\delta\right).\label{bosup}
\end{align}
Also we recall that $r \in (0,1)$ therefore we can prove easily using Laplace's method that $\lim_{\alpha \rightarrow + \infty}\widetilde{s}_\alpha(a_{r})=\lim_{\alpha \rightarrow + \infty}\widetilde{s}_\alpha(b_{r})=0$. We conclude by noticing that $P-a.s.,\ \limab\inf_{y\in[\widetilde{s}_\alpha(a_{r}),\widetilde{s}_\alpha(b_{r})]}L_B(\sB(1),y)=1$, thanks to the continuity of the function $y\rightarrow L_B(\sB(1),y)$ at $0$.

 \emph{We continue with the upper bound}.  First we use the same idea as in the proof of Lemma \ref{eqtm}: we establish that with a probability which tends to $1$ as $\alpha$ goes to infinity, $X_\alpha$ does not exit from the standard valley $(p,m,q)$. Define
\begin{displaymath}
	L:=\inf\left\{x\leq 0/L_B(\sB(1),x)>0\right\}\text{ and }U:=\sup\left\{x\geq 0/L_B(\sB(1),x)>0\right\},
\end{displaymath}
we claim that, 
\begin{equation}
P\text{-a.s.}\ \exists \alpha_0,\ \forall\alpha>\alpha_0,\ p<\widetilde{s}_\alpha^{-1}(L)<0<\widetilde{s}_\alpha^{-1}(U)<q.\label{equinf}
\end{equation}
Indeed, we have 
\begin{displaymath}
	\widetilde{s}_\alpha^{-1}(L)\geq p\Longleftrightarrow L\geq \widetilde{s}_\alpha(p)=-e^{-\alpha h(\alpha)}\int_{p}^0e^{\alpha W(x)}\ud x,
\end{displaymath}
and by Laplace's method we get $\widetilde{s}_\alpha(p)=-e^{\alpha(W(p)-h(\alpha))+o(\alpha)}$.
It follows from the fact that $W \in \tW$ and $\limab h(\alpha)=1<D\leq W(p)$ that $\lima \widetilde{s}_\alpha(p)=-\infty,\ P\textrm{-a.s.}$. In a similar way we obtain	$\lima \widetilde{s}_\alpha(q)=+\infty,\ P\textrm{-a.s.}$ and $(\ref{equinf})$ is satisfied. 

On the event $\left\{p<\widetilde{s}_\alpha^{-1}(L)<0<\widetilde{s}_\alpha^{-1}(U)<q\right\}$, we can write
\begin{align*}
	\int_{-\infty}^{+\infty}e^{-\alpha W(x)}L_B(\sB(1),\widetilde{s}_\alpha(x))\ud x&=\int_{p}^{a_{r}}e^{-\alpha W(x)}L_B(\sB(1),\widetilde{s}_\alpha(x))\ud x\\
	&+\int_{a_{r}}^{b_{r}}e^{-\alpha W(x)}L_B(\sB(1),\widetilde{s}_\alpha(x))\ud x\\
	&+\int_{b_{r}}^{q}e^{-\alpha W(x)}L_B(\sB(1),\widetilde{s}_\alpha(x))\ud x.
\end{align*}
We only have to find an upper bound for these integrals. First, we have
\begin{align*}
	&\int_{p}^{a_{r}}e^{-\alpha W(x)}L_B(\sB(1),\widetilde{s}_\alpha(x))\ud x+\int_{b_{r}}^{q}e^{-\alpha W(x)}L_B(\sB(1),\widetilde{s}_\alpha(x))\ud x\\
	&\leq(q-p)\exp(-\alpha\min_{x\in I_r} W(x))\sup_{y\in\R}L_B(\sB(1),y)
\end{align*}
where $I_r:=[p,a_{r}]\cup[b_{r},q]$. Moreover
\begin{align*}
	\int_{a_{r}}^{b_{r}}e^{-\alpha W(x)}L_B(\sB(1),\widetilde{s}_\alpha(x))\ud x 
&\leq\sup_{y\in[\widetilde{s}_\alpha(a_{r}),\widetilde{s}_\alpha(b_{r})]}L_B(\sB(1),y)\int_{a_{r}}^{b_{r}}e^{-\alpha W(x)}\ud x\\
&=g(\alpha) \sup_{y\in[\widetilde{s}_\alpha(a_{r}),\widetilde{s}_\alpha(b_{r})]}L_B(\sB(1),y).\\
\end{align*}
Therefore, assembling the last two inequalities and the equality in law $(\ref{eglsa})$
\begin{align}
&P\left(\frac{\sa}{\gW e^{\alpha h(\alpha)}}-1\leq\delta\right)\geq\nonumber\\
&P\left(\sup_{y\in[\widetilde{s}_\alpha(a_{r}),\widetilde{s}_\alpha(b_{r})]}L_B(\sB(1),y)-1+\frac{(q-p)}{\gW}e^{-\alpha\min_{I_r}W}\sup_{y\in\R}L_B(\sB(1),y)\leq\delta\textrm{ ;}\right.\nonumber\\
&p<\widetilde{s}_\alpha^{-1}(L)<0<\widetilde{s}_\alpha^{-1}(U)<q\Bigg).\label{boinf} 
\end{align}
By hypothesis $r<1$, so $\limab \widetilde{s}_\alpha(a_{r})=\limab\widetilde{s}_\alpha(b_{r})=0$, moreover $y\rightarrow L_B(\sigma_B(1),y)$ is $P$-a.s. continuous at $0$, it follows
\begin{displaymath}
 \lima\sup_{y\in[\widetilde{s}_\alpha(a_{r}),\widetilde{s}_\alpha(b_{r})]}L_B(\sB(1),y)=1\quad P\text{-a.s.}.
\end{displaymath}
We also know that $\limab \frac{\log g(\alpha)}{\alpha}=0$, moreover according to the definition of  $\tW$, $\underset{x\in I_r}{\min}W(x)>0$, and finally $\sup_{y\in\R}L_B(\sB(1),y)$ is $P$-a.s. finite, thus
\begin{displaymath}
 \lima\frac{(q-p)}{\gW}e^{-\alpha\min_{I_r}W}\sup_{y\in\R}L_B(\sB(1),y)=0\quad P\textrm{-a.s.}.
\end{displaymath}
Putting the last two assertions together with (\ref{boinf}) we get the upper bound and finally the lemma.
\end{proof}

\section{Proof of the main results  \label{S3}}
One of the key result of this paper is Theorem \ref{limL}, the other results can be deduced from that theorem together with estimates on the random environment, so we naturally start with the 
\subsection{Proof of Theorem \ref{limL}  \label{S3.1}}

We begin with a proposition which resume Propositions \ref{limLT} and \ref{limT}, we get the asymptotic behavior of $L_{X_\alpha}$ within a deterministic interval of time:
\begin{proposition}\label{cle}
Let $K>0$, $r\in(0,1)$, $W\in\tW$ and $h$ a real function such that $\underset{\alpha\rightarrow\infty}{\lim}h(\alpha)=1$. For all $\delta>0$, we have
\begin{displaymath}
\lima P\left(\sup_{-K\leq x\leq K}\left|\frac{L_{X_\alpha}(e^{\alpha h(\alpha)},m+\ax)}{e^{\alpha h(\alpha)}}\frac{\int_{a_{r}}^{b_{r}}e^{-\alpha W_m(y)}\ud y}{e^{-\alpha W_m(\ax)}}-1\right|\leq \delta \right)=1.
\end{displaymath}
\end{proposition}
\begin{proof}
Let $\delta>0$, and $f:\R^+\rightarrow \R^+$ such that $\limab f(\alpha)=1$, define
\begin{displaymath}
A_{\alpha,f}:=\left\{\sup_{-K\leq x\leq K}\left|\frac{L_{X_\alpha}(\tla(e^{\alpha f(\alpha)},m),m+\ax)}{e^{\alpha f(\alpha)-\alpha W_m(\ax)}}-1\right|\leq \delta\right\}
\end{displaymath}
and
\begin{displaymath}
B_{\alpha,f}:=\left\{\left|\frac{\tla(e^{\alpha f(\alpha)},m)}{\gW e^{\alpha f(\alpha)}}-1\right|\leq \delta\right\}
\end{displaymath}
where, as in the previous section, $g(\alpha)=\int_{a_{r}}^{b_{r}}e^{-\alpha W_m(y)}\ud y$.
We also define two functions
\begin{align*}
h^+(\alpha)&:=h(\alpha)-\alpha^{-1}\log\left(\gW(1-\delta)\right),\\
h^-(\alpha)&:=h(\alpha)-\alpha^{-1}\log\left(\gW(1+\delta)\right).
\end{align*}
On $B_{\alpha,h^+}$ the following inequality holds:
\begin{align*}
	\tla(e^{\alpha h^+(\alpha)},m)&\geq(1-\delta)\gW e^{\alpha h^+(\alpha)},\\
	&\geq e^{\alpha h(\alpha)},
\end{align*}
moreover in its first coordinate the local time is an increasing function, therefore on $A_{\alpha,h^+}\cap B_{\alpha,h^+}$, $\forall x\in [-K,K]$,
\begin{align*}
	L_{X_\alpha}(e^{\alpha h(\alpha)},m+\ax)&\leq L_{X_\alpha}(\tla(e^{\alpha h^+(\alpha)},m),m+\ax),\\
	&\leq e^{\alpha h^+(\alpha)-\alpha W_m(\ax)}(1+\delta),\\
	&\leq\frac{e^{\alpha h(\alpha)-\alpha W_m(\ax)}}{\gW}\frac{1+\delta}{1-\delta}.
\end{align*}
In the same way, on $A_{\alpha,h^-}\cap B_{\alpha,h^-}$ we obtain
\begin{displaymath}
	L_{X_\alpha}(e^{\alpha h(\alpha)},m+\ax)\geq\frac{e^{\alpha h(\alpha)-\alpha W_m(\ax)}}{\gW}\frac{1-\delta}{1+\delta}.
\end{displaymath}
By Laplace's method, $\limab \frac{\log g(\alpha)}{\alpha}=0$, so $h^+$ and $h^-$ tend to $1$ when $\alpha$ goes to infinity and we can apply Propositions \ref{limLT} and \ref{limT}, finally 
\begin{displaymath}
\lima P\left(A_{\alpha,h^+}\cap B_{\alpha,h^+}\cap A_{\alpha,h^-}\cap B_{\alpha,h^-}\right)=1	
\end{displaymath}
and the proposition is proved.
\end{proof}

 We turn back to the proof of the theorem, notice that the difference between  Theorem \ref{limL} and Proposition \ref{cle} above is the process itself: one deals with $X$ whereas the other deals with $X_{\alpha}$. To finish the proof we need to show that thanks to Lemma \ref{egx}, the theorem can be deduced from the proposition. 

Let $\alpha>0$ and recall that $W^\alpha(.):=\alpha^{-1}W(\alpha^2\cdot)$. First, remark that for all $W\in\mathcal{W}$, 
\begin{align*}
m_1(W^\alpha)&=\alpha^{-2}m_\alpha(W),\\
 a_{ r}(W^\alpha_{m_1(W^\alpha)})&=\alpha^{-2}a_{\alpha r}(W_{m_\alpha(W)}),\\
 b_{r}(W^\alpha_{m_1(W^\alpha)})&=\alpha ^{-2} b_{\alpha r}(W_{m_\alpha(W)}),
\end{align*}
and for any $x\in\R$,
\begin{displaymath}
W^\alpha_{m_1(W^\alpha)}(x)=\frac{1}{\alpha}W_{m_\alpha(W)}(\alpha^2x).
\end{displaymath}
Now replacing $t$ by $\alpha^{-4}e^{\alpha}$ in the second part of Lemma \ref{egx}, we obtain for every $W\in\mathcal{W}$,
\[
\left(L_{X(\alpha W^\alpha,\cdot)}(\alpha^{-4}e^{\alpha},m_1(W^\alpha)+\ax)\right)_{x\in\R}\eglw\left(\frac{1}{\alpha^2}L_{X(W,\cdot)}(e^\alpha,m_\alpha(W)+x)\right)_{x\in\R}.
\]
Therefore, for any $\alpha>0$, $\delta>0,K>0$ and $W\in \mathcal{W}$,
\begin{align*}
&P\left(\sup_{-K\leq x\leq K}\left|\frac{L_{X(W,\cdot)}(e^{\alpha},m_\alpha(W)+x)}{e^{\alpha}}\frac{\int_{a_{\alpha r}}^{b_{\alpha r}}e^{-W_{m_\alpha}(y)}\ud y}{e^{-W_{m_\alpha}(x)}}-1\right|<\delta\right)\quad=\\
&P\left(\sup_{-K\leq x\leq K}\left|\frac{L_{X(\alpha W^\alpha,\cdot)}(\alpha^{-4}e^{\alpha},m_1(W^\alpha)+\ax)}{\alpha^{-2}e^{\alpha}}\frac{\int_{\alpha^2a_{r}}^{\alpha^2b_{r}}e^{-\alpha W^\alpha_{m_1}(\alpha^{-2}y)}\ud y}{e^{-\alpha W^\alpha_{m_1}(\alpha^{-2}x)}}-1\right|<\delta\right).
\end{align*} 
Moreover, for all $\alpha>0$, $\mathcal{P}$ is invariant under the transformation $W\mapsto W^\alpha$, we get that 
\begin{align*}
&\Pp\left(\sup_{-K\leq x\leq K}\left|\frac{L_{X(W,\cdot)}(e^{\alpha},m_\alpha(W)+x)}{e^{\alpha}}\frac{\int_{a_{\alpha r}}^{b_{\alpha r}}e^{-W_{m_\alpha}(y)}\ud y}{e^{-W_{m_\alpha}(x)}}-1\right|<\delta\right)\quad=\\
&\Pp\left(\sup_{-K\leq x\leq K}\left|\frac{L_{X(\alpha W,\cdot)}(\alpha^{-4}e^{\alpha},m_1(W)+\ax)}{\alpha^{-4}e^{\alpha}}\frac{\int_{a_{r}}^{b_{r}}e^{-\alpha W_{m_1}(y)}\ud y}{e^{-\alpha W_{m_1}(\ax)}}-1\right|<\delta\right)
\end{align*}
and we recall that $\Pp=\mathcal{P} \otimes P$. To finish the proof we notice that $\mathcal{P}(\tW)=1$, $\alpha^{-4}e^{\alpha}=e^{\alpha(1-\frac{4}{\alpha}\log\alpha)}$ and $\limab(1-\frac{4}{\alpha}\log\alpha)=1$, so applying Proposition \ref{cle} we get Theorem \ref{limL}.
\hfill $\square$ 
\vspace{0.2cm}

As we said at the beginning of the section, once Theorem \ref{limL} is proved, we get the other results by studying in details some properties of the random environment. 

\subsection{Proof of Theorems \ref{thlimsup} and \ref{cvloi} \label{S3.3}}
Denote $m^*(t):=\min\{x\in\R,\ L_X(x,t)=L_X^*(t)\}$ the (first) favorite point of the diffusion at time $t$. D. Cheliotis has shown in \cite{Cheliotis} that $m^*(e^\alpha)-m_{\alpha}$ converges in $\Pp$ probability to $0$.
Thus for any $K>0$, the processes $L_X^*(e^\alpha)/e^\alpha$ and $\sup_{x\in[-K,K]}L_X(e^\alpha,m_\alpha+x)/e^\alpha$ have the same limit in law and Theorem \ref{thlimsup} is now a straight forward consequence of Theorem  \ref{cvloi}. So we are left to prove Theorem \ref{cvloi}. The main elements of the proof are Theorem \ref{limL} and Lemmata \ref{eqint} and \ref{cvloiw} below. 
 We begin with the proof of the first lemma:
\begin{lemme}\label{eqint}
For all $r \in (0,1)$, $\mathcal{P}$-almost surely for all $a_r\leq a< 0< b\leq b_r$, we get
$$\int_{a_r}^{b_r} e^{-\alpha W_{m_1}(y)}dy \underset{\alpha\rightarrow +\infty}{\sim} \int_a^b e^{-\alpha W_{m_1}(y)}dy.$$
\end{lemme}
\begin{proof}
For any $W\in\mathcal{W}$,
\begin{align*}
\int_{a_r}^{b_r} e^{-\alpha W_{m_1}(y)}dy - \int_a^b e^{-\alpha W_{m_1}(y)}dy &= \int_{a_r}^a e^{-\alpha W_{m_1}(y)}dy + \int_b^{b_r} e^{-\alpha W_{m_1}(y)}dy\\
&\leq(b_r-a_r)e^{-\alpha \min_{I}W_{m_1}}
\end{align*}
where $I:=[a_r,a]\cup[b,b_r]$.
Moreover, by Laplace's method,
$$\lim_{\alpha\rightarrow +\infty} \frac{1}{\alpha} \log \int_{a_r}^{b_r} e^{-\alpha W_{m_1}(y)}dy = -\min_{[a_r,b_r]}(W_{m_1}).$$
As $a_r\leq 0\leq b_r$, this minimum is equal to $0$, thus $\int_{a_r}^{b_r} e^{-\alpha W_{m_1}(y)}dy = e^{o(\alpha)}$ and so,
$$\frac{\int_{a_r}^{b_r} e^{-\alpha W_{m_1}(y)}dy-\int_a^b e^{-\alpha W_{m_1}(y)}dy}{\int_{a_r}^{b_r} e^{-\alpha W_{m_1}(y)}dy} \leq (b_r-a_r)e^{-\alpha\min_{I}W_{m_1}+o(\alpha)}.$$
And for any $W\in\tW$, $\min_{I}W_{m_1}>0$, then letting $\alpha$ go to infinity we obtain the equivalence for any $W\in\tW$. As $\mathcal{P}(\tW)=1$, this implies the result of the lemma.
\end{proof}
Now, the proof of the theorem will be finished once we will have shown the following lemma.
\begin{lemme}\label{cvloiw}
For any positive constant $K$ and for any bounded continuous functional $F$ on $C(\R,\R)$ such that $F(f)$ depends only on the values of the function $f$ on $[-K,K]$,
\begin{equation}
 \lima\mathcal{E}\left[F\left(\frac{e^{-W_{m_\alpha}}}{\int_{a_{\alpha r}}^{b_{\alpha r}}e^{-W_{m_\alpha}(y)}\ud y}\right)\right]=\mathcal{E}\left[F\left(\frac{e^{-R}}{\int_{-\infty}^{\infty}e^{-R(y)}\ud y}\right)\right].
\end{equation}
where $\mathcal{E}$ stands for the expectation with respect to the probability measure $\mathcal{P}$ and
\begin{displaymath}
	\forall x\in\R,\ R(x):=R_1(x)\textbf{\emph{1}}_{\left\{x\geq0\right\}}+R_2(-x)\textbf{\emph{1}}_{\left\{x<0\right\}},
\end{displaymath}
$R_1$ and $R_2$ being two independent 3-dimensional Bessel processes.
\end{lemme}

\begin{proof}
According to the scaling property of the Brownian motion,
\begin{equation*}
 \mathcal{E}\left[F\left(\frac{e^{-W_{m_\alpha}(\cdot)}}{\int_{a_{\alpha r}}^{b_{\alpha r}}e^{-W_{m_\alpha}(y)}\ud y}\right)\right]=\mathcal{E}\left[F\left(\frac{e^{-\alpha W_{m_1}(\alpha^{-2}\cdot)}}{\alpha^2\int_{a_{r}}^{b_{r}}e^{-\alpha W_{m_1}(y)}\ud y}\right)\right].
\end{equation*}
So, thanks to Lemma \ref{eqint}, denoting
\begin{displaymath}
\tilde{a}_r:=\left\{
\begin{array}{lc}
a_r&\textrm{if } m_1<0\\
a_r\vee0& \textrm{if } m_1\geq0 
\end{array}\right.
 \textrm{and } \tilde{b}_r:=\left\{
\begin{array}{lc}
 b_r& \textrm{if } m_1\geq0\\
b_r\wedge0& \textrm{if } m_1<0
\end{array}\right.,
\end{displaymath}
it is enough to prove that 
\begin{equation*}
 \lima\mathcal{E}\left[F\left(\frac{e^{-\alpha W_{m_1}(\alpha^{-2}\cdot)}}{\alpha^2\int_{\tilde{a}_{r}}^{\tilde{b}_{r}}e^{-\alpha W_{m_1}(y)}\ud y}\right)\right]=\mathcal{E}\left[F\left(\frac{e^{-R}}{\int_{-\infty}^{\infty}e^{-R(y)}\ud y}\right)\right].
\end{equation*}
This can be done using the same arguments as in the proof of Lemma 3.2 in Tanaka \cite{Tanaka}. 
\end{proof}

\noindent \\ \textbf{ Acknowledgment : } We would like to thank Romain Abraham for very helpful discussions, many remarks and comments on the different versions of this paper.

\bibliographystyle{plain} 
 \bibliography{thbiblio}
 
\end{document}